\documentclass[11pt,reqno]{amsart}
\usepackage{mathrsfs}
\usepackage{url}
\usepackage{mathtools}
\usepackage{latexsym,epsfig,amssymb,amsmath,amsthm,color,url}
\usepackage[inline,shortlabels]{enumitem}
\usepackage{hyperref}
\usepackage[foot]{amsaddr}
\RequirePackage[numbers]{natbib}

\allowdisplaybreaks 
 \numberwithin{equation}{section}
\newtheorem{theorem}{Theorem}[section]
\newtheorem{prop}[theorem]{Proposition}
\newtheorem{lemma}[theorem]{Lemma}

\renewcommand{\theenumi}{\roman{enumi}}

\newcommand\Item[1][]{%
  \ifx\relax#1\relax  \item \else \item[#1] \fi
  \abovedisplayskip=0pt\abovedisplayshortskip=0pt~\vspace*{-\baselineskip}}

\theoremstyle{definition}

\theoremstyle{remark}
\newtheorem{rem}[theorem]{Remark}

\theoremstyle{conjecture}

\DeclareMathOperator{\Prob}{\mathbf{P}}

\DeclareMathOperator{\tv}{TV}
\DeclareMathOperator{\bin}{Bin}
\DeclareMathOperator{\poi}{Poisson}

\title[Uniqueness of RSBM communities]{Uniqueness of communities in regular stochastic block models}
\date{}
\author{Sayar Karmakar, Moumanti Podder}
\address{Sayar Karmakar, Department of Statistics, University of Florida, Gainesville, FL 32611, United States.}
\email{sayarkarmakar@ufl.edu}
\address{Moumanti Podder, Department of Mathematics, Indian Institute of Science Education and Research (IISER) Pune, Dr. Homi Bhabha Road, Pashan, Pune 411008, India.}
\email{moumantip3@gmail.com}

\begin{document}
\bibliographystyle{plainnat}

\begin{abstract}
This paper studies the regular stochastic block model comprising \emph{several} communities: each of the $k$ non-overlapping communities, for $k \geqslant 3$, possesses $n$ vertices, each of which has total degree $d$. The values of the intra-cluster degrees (i.e.\ the number of neighbours of a vertex inside the cluster it belongs to) and the inter-cluster degrees (i.e.\ the number of neighbours of a vertex inside a cluster different from its own) are allowed to vary across clusters. We discuss two main results: the first compares the probability measure induced by our model with the uniform measure on the space of $d$-regular graphs on $kn$ vertices, and the second establishes that the clusters, under rather weak assumptions, are unique asymptotically almost surely as $n \rightarrow \infty$.  
\end{abstract}



\thanks{The second author was partially supported by the grant NSF DMS-1444084.}

\maketitle

\section{Introduction}\label{intro}
\sloppy This paper concerns itself with the \emph{regular stochastic block model}, henceforth abbreviated as RSBM, that is used to study \emph{clustered networks}. These networks exhibit \emph{community structure}, whereby the individuals participating in the network, typically indicated as nodes or vertices of a graph, are split into overlapping or non-overlapping groups, usually with dense connections internally and sparser connections between different groups. Community structure is common in many complex networks such as computer and information networks (\cite{statistical_SBM}), online social networks and biological networks (\cite{girvan_newman, choudhury_paul, fani_bagheri, moosavi}) that include protein-protein and gene-gene interactions (\cite{genetic}), biological neural networks (\cite{brain}), metabolic networks (\cite{metabolic}) etc. Detecting communities in clustered networks has been pursued with fervour (\cite{newman_1, newman_fast, newman_2, fast_unfolding, nadakuditi_newman, karrer_newman}), since communities often act as meta-nodes in a network and individuals within the same community tend to exhibit behavioural and functional similarities, simplifying the analysis of the underlying features of the network. The characteristics displayed by each distinct community may also vary greatly from the average properties of the network. The existence of communities may also significantly affect the spreading of rumours, epidemics etc.\ within the network.

The \emph{stochastic block model}, henceforth abbreviated as SBM, (introduced in \cite{holland}, surveyed in \cite{survey}), has been the most popular model, so far, in studying clustered networks. In its most simplified form, this model comprises $2n$ vertices that are partitioned into two equi-sized clusters. Edges between all pairs of vertices appear mutually independently, with probability $p$ if both vertices belong to the same cluster, and probability $q$ if they belong to different clusters. Letting the intra-cluster average degree be $a \sim p n$ and the inter-cluster average degree be $b \sim q n$, \cite{mcsherry} and \cite{exact} studied the SBM in the regime where $a, b = O(\log n)$, whereas \cite{coja}, \cite{decelle_1}, \cite{decelle_2}, \cite{mossel_1}, \cite{mossel_2}, \cite{mossel_3}, \cite{mossel_4} and \cite{massoulie_ramanujan} studied SBM in the regime where $a, b = O(1)$, as $n$ grows to $\infty$. Other variants of this model studied are the Bayesian SBM (\cite{peixoto}), degree-corrected block models (\cite{yan}), labeled SBM (\cite{labeled_1, labeled_2}), SBM in sparse hypergraphs (\cite{pal_zhu}) etc. We also refer the reader to \cite{dyer_frieze}, \cite{jerrum_sorkin} and \cite{condon} for discussions on relations between community detection in SBM and the minimum bisection problem that seeks to partition a graph of $2n$ vertices into two equi-sized parts such that the number of edges across the parts is minimized. 

Since the essence of our paper is to focus on the case where the given model constitutes \emph{several} underlying clusters, we emphasize on the following developments in the literature. In \cite{decelle_1}, it was conjectured that if the signal-to-noise ratio of a given SBM is strictly higher than $1$, then it is possible to detect communities in polynomial time, or, in other words, the well-known Kesten-Stigum threshold is achieved; moreover, if the number of underlying communities in the model exceeds $4$, it is possible to detect the communities information-theoretically for some signal-to-noise ratio strictly lower than $1$. It was shown in \cite{bordenave} that the Kesten-Stigum threshold is achieved in SBM's with multiple communities satisfying certain asymmetry assumptions, whereas the full conjecture of \cite{decelle_1}, for several clusters, was established in \cite{abbe_sandon}. The extension of SBM from two to several communities has proven to be a veritable challenge during the course of development of this field.

The RSBM was introduced in \cite{09} (note that two regular versions of the SBM in the sparse regime was proposed in \cite{mossel_1}, and \emph{equitable random graphs} have been studied in \cite{newman_martin} and \cite{barucca}). As in \cite{09}, we assume that each intra-cluster degree and each inter-cluster degree exceeds 3, ensuring that the resulting graph is connected with high probability. The RSBM differs from the SBM with constant average intra-cluster and inter-cluster degrees in that the latter has a positive probability of possessing isolated vertices. In RSBM, the imposition of the constraint that each vertex has a constant number of neighbours in each given cluster gives more structure to the graph, but at the same time, robs the model of the edge-independence that is present in SBM. 

We now highlight the novelties as well as describe the organization of our paper. We emphasize that our model is more general than that of \cite{09} in that it takes into account multiple communities as well as intra-cluster and inter-cluster degree values that differ across communities. We answer questions similar to those in \cite{09}, but the proof techniques are more involved and require careful analysis. To set the stage, in \S\ref{notation}, we describe the notations and terminology used throughout the paper; in \S\ref{model_defn}, we describe the model and its underlying measure in details; in \S\ref{configuration:subsec}, we describe the well-known configuration model and the associated exploration process, and their importance in the generation of uniformly random regular or bipartite-regular graphs. 

In \S\ref{sec:measure_diff}, we show that the measure induced by RSBM on $k n$ vertices, each with degree $d$, where $k$ denotes the number of communities and $n$ the number of members in each community, is distinct from the measure that makes a uniformly random selection out of the collection of all $d$-regular graphs on $k n$ vertices. In \S\ref{sec:unique}, we show that under rather weak assumptions, the underlying clusters of the model are unique almost surely as $n$ approaches $\infty$. We draw attention of the reader to a key difference between our analysis and the analysis in [\cite{09}, \S 3.2.2]: while they had the symmetry, around $1/2$, of the binary entropy function $H(\alpha) = -\alpha \log_{2} \alpha - (1-\alpha) \log_{2} (1-\alpha)$ in their favour due to the presence of only two clusters, we require a somewhat different strategy to handle the higher number of clusters in our model. Even in the homogeneous case, where all the intra-cluster degrees are the same, there is need for a thorough case-by-case analysis that is much more intricate than in \cite{09}. We emphasize here that in the homogeneous case, our analysis allows for the inter-cluster degrees to exceed, by far, the intra-cluster degrees. This is a significant generalization over the much more usual assumption of denser intra-cluster connections and sparser inter-cluster connections found in the literature. In the heterogeneous scenario, we need a more restricted range of intra-cluster and inter-cluster degrees, as described in [\eqref{not_all_intra_degrees_same_cond'} and \eqref{not_all_intra_degrees_same_cond}, Theorem~\ref{main_2}]. 

\subsection{Notations}\label{notation} Given $n, d \in \mathbb{N}$, we denote by $\mathcal{R}^{n}_{d}$ the set of all $d$-regular graphs on $n$ labeled vertices, and by $\mathcal{B}^{n}_{d}$ the set of all $d$-bipartite-regular graphs on $2n$ labeled vertices where each cluster comprises $n$ vertices. We shall denote by $\mu^{n}_{d}$ the uniform measure on $\mathcal{R}^{n}_{d}$. 

Given a graph $G$, we denote by $V(G)$ its vertex set and by $E(G)$ its edge set. Given $S \subset V(G)$ and $v \in V(G)$, we let $\deg_{S}(v)$ denote the number of edges $\{u, v\}$ where $u \in S$. We denote by $G|_{S}$ the subgraph of $G$ that is induced on $S$. For disjoint subsets $S_{1}$ and $S_{2}$ of $V(G)$, we denote by $\deg(S_{1}, S_{2})$ the number of edges $\{u_{1}, u_{2}\}$ where $u_{1} \in S_{1}$ and $u_{2} \in S_{2}$. We let $G|_{S_{1}, S_{2}}$ denote the subgraph with vertex set $S_{1} \cup S_{2}$ and edge set $\left\{\{u_{1}, u_{2}\}: u_{1} \in S_{1}, u_{2} \in S_{2}\right\}$. Given $v \in V(G)$ and $r \in \mathbb{N}$, we let $B(v, r) = \left\{u \in V(G): \rho(u, v) \leqslant r\right\}$ be the neighbourhood of radius $r$ around $v$, where $\rho$ is the usual graph metric. We let $\delta B(v, r) = \left\{u \in V(G): \rho(u, v) = r\right\}$ denote the boundary of $B(v, r)$. 

Given an infinite sequence of graphs $\left\{G_{n}\right\}$ and a graph property $A$, we say that $A$ holds \emph{asymptotically almost surely} (a.a.s.)\ for this sequence if $\Prob[G_{n} \text{ satisfies property } A] \rightarrow 1$ as $n \rightarrow \infty$. Given any two functions $f: \mathbb{N} \rightarrow \mathbb{R}^{+}$ and $g: \mathbb{N} \rightarrow \mathbb{R}^{+}$, we write $f(n) \sim g(n)$ to indicate that $f(n)/g(n) \rightarrow 1$ as $n \rightarrow \infty$. 

For any $\alpha \in (0,1)$, recall that the Shannon entropy for a Bernoulli$(\alpha)$ distribution is given by $H(\alpha) = -\alpha \log_{2} \alpha - (1-\alpha) \log_{2} (1-\alpha)$. This will be used in \S\ref{sec:unique}.

\subsection{Decription of the model}\label{model_defn} Our $k$-cluster RSBM, denoted $\mathcal{G}^{n}_{A}$, has the following parameters: 
\begin{enumerate}
\item $n$ denotes the number of vertices in each cluster, 
\item $k$ denotes the number of clusters, 
\item and $A = \left(A_{i,j}\right)_{1 \leqslant i, j \leqslant k}$ is a $k \times k$ symmetric matrix of strictly positive integers such that, for some $d \in \mathbb{N}$,
\begin{equation}\label{degree_criterion}
\sum_{j=1}^{k} a_{i,j} = d \text{ for all } i = 1, \ldots, k.
\end{equation} 
\end{enumerate}
Starting with $k n$ labeled vertices, we uniformly randomly partition them into $k$ clusters $\mathcal{C}_{1}, \ldots, \mathcal{C}_{k}$, each of size $n$. Independent of each other, we now place on the vertices of $\mathcal{C}_{i}$ a uniformly random member of $\mathcal{R}^{n}_{a_{i,i}}$, and across the clusters $\mathcal{C}_{i}$ and $\mathcal{C}_{j}$ a uniformly random member of $\mathcal{B}^{n}_{a_{i,j}}$, for all $i, j \in \{1, \ldots, k\}$ with $i \neq j$. The criterion in \eqref{degree_criterion} ensures that any realization of our model will be $d$-regular. All the parameters except $n$ remain fixed throughout our analysis. We analyze the asymptotic behaviour of the model as $n \rightarrow \infty$.

\subsection{Configuration model and exploration process}\label{configuration:subsec} The \emph{configuration model} plays a crucial role as a tool in our arguments in \S\ref{sec:unique}. Given $d, n \in \mathbb{N}$ such that $d n$ is even, this model (see \cite{07, 08}) allows us to generate a $d$-regular random graph on $n$ labeled vertices $v_{1}, \ldots, v_{n}$ (possibly with self-loops and parallel edges) according to the following procedure, also known as the \emph{exploration process}:
\renewcommand{\theenumi}{\roman{enumi}}
\begin{enumerate}
\item Fix a total order $v_{1} < v_{2} < \ldots < v_{n}$ on the vertex set, and let $\Xi_{i} = \left\{\xi_{i,j}: 1 \leqslant j \leqslant d\right\}$ denote the set of \emph{half-edges} emanating from $v_{i}$, for all $i = 1, \ldots, n$. Let $\Xi = \bigcup_{i=1}^{n} \Xi_{i}$. We define a total ordering on $\Xi$ as follows: all half-edges in $\Xi_{i}$ come before every half-edge in $\Xi_{i+1}$ for all $i = 1, \ldots, n-1$, and within each $\Xi_{i}$, we have $\xi_{i,j} < \xi_{i,j+1}$ for all $1 \leqslant j \leqslant d-1$.
\item We first choose $\hat{\xi}$ uniformly randomly from the set $\Xi \setminus \{\xi_{1,1}\}$ and form the edge $\left\{\hat{\xi}, \xi_{1,1}\right\}$. Having constructed the first $k$ edges, we find the smallest half-edge $\xi_{i,j}$ yet unmatched with another half-edge, and choosing a $\tilde{\xi}$ uniformly randomly from the remaining subset of half-edges, we form the edge $\left\{\xi_{i,j}, \tilde{\xi}\right\}$. Thus we form a perfect matching on $\Xi$. 
\end{enumerate}

We also describe here the exploration process aimed at generating a random $d$-regular bipartite graph in which each cluster contains $n$ vertices. In this case, we label the vertices of one partition as $u_{1} < u_{2} < \cdots < u_{n}$ and the other as $v_{1} < v_{2} < \cdots < v_{n}$. We let $\Xi_{i} = \left\{\xi_{i,j}: 1 \leqslant j \leqslant d\right\}$ denote the set of half-edges emanating from $u_{i}$ and $\Gamma_{i} = \left\{\gamma_{i,j}: 1 \leqslant j \leqslant n\right\}$ the set of half-edges emanating from $v_{i}$, and the total orderings on $\Xi = \bigcup_{i=1}^{n} \Xi_{i}$ and on $\Gamma = \bigcup_{i=1}^{n} \Gamma_{i}$ are analogous to the one described above. We first choose, uniformly randomly, a $\hat{\gamma}$ out of $\Gamma$, and form the edge $\{\hat{\gamma}, \xi_{1,1}\}$. After having constructed the $k$-th edge, we find the smallest $\xi_{i,j}$ in $\Xi$ that is yet to be matched with a half-edge from $\Gamma$. We choose, uniformly randomly, a half-edge $\tilde{\gamma}$ from $\Gamma$ that has not yet been matched, and form the edge $\{\tilde{\gamma}, \xi_{i,j}\}$. This leads to a perfect matching between $\Xi$ and $\Gamma$. 

It has been shown in \cite{08} that in either of the cases above, the probability that the generated random graph is simple, i.e.\ devoid of self-loops and parallel edges, stays bounded away from $0$ as $d$ stays bounded and $n \rightarrow \infty$. We can thus condition on the event that the generated random graph is simple, which in turn allows us to prove the results in \S\ref{sec:unique} using the exploration process. 

Henceforth, we call a half-edge emanating from a vertex in $\mathcal{C}_{i}$ and matched with a half-edge from a vertex in $\mathcal{C}_{j}$, a \emph{half-edge of type $\{i, j\}$}, for all $i, j \in \{1, \ldots, k\}$. 

\section{Comparing RSBM with uniform measure on $d$-regular graphs on $kn$ vertices}\label{sec:measure_diff}
We state here our first main result. Let $\mu^{n}_{A}$ denote the probability measure of $\mathcal{G}_{n}^{A}$ and $\mathcal{S}^{n}_{A}$ the support of $\mu^{n}_{A}$. Recall from \S\ref{notation} that $\mu^{kn}_{d}$ denotes the uniform random measure on $\mathcal{R}^{kn}_{d}$. 

\begin{theorem}\label{main_1}
Under the above set-up, keeping the matrix $A$ fixed, we have
\begin{equation}
\lim_{n \rightarrow \infty} \left|\left|\mu^{n}_{A}, \mu^{kn}_{d}\right|\right|_{\tv} = 1,
\end{equation}
where $\tv$ denotes the total variation distance between two probability measures.
\end{theorem}

The proof begins with stating two well-known results. For given $n$ and $d$ with $1 \leqslant d = o\left(n^{1/2}\right)$, [\cite{01}, Corollary 5.3] states that
\begin{equation}\label{regular_graph_count}
\left|\mathcal{R}^{n}_{d}\right| = C \frac{(nd)!}{(nd/2)! 2^{nd/2}(d!)^{n}},
\end{equation}
where $C = C(n, d)$ remains bounded as $n$ grows. Similarly, [\cite{02}, Theorem 2] states that 
\begin{equation}\label{bipartite_regular_graph_count}
\left|\mathcal{B}^{n}_{d}\right| = C' \frac{(dn)!}{(d!)^{2n}},
\end{equation}
where $C' = C'(n,d)$ remains bounded as $n$ grows. We now use these estimates to count the total number of possible realizations of $\mathcal{G}^{n}_{A}$. 
Given $kn$ labeled vertices, we choose the vertex sets for the clusters $\mathcal{C}_{1}, \ldots, \mathcal{C}_{k}$ in 
\begin{align}
{kn \choose n} {(k-1)n \choose n} \cdots {2n \choose n} {n \choose n} &= \frac{(kn)!}{\left(n!\right)^{k}} \sim \frac{\sqrt{2 \pi (k n)} \left(\frac{k n}{e}\right)^{k n}}{\left\{\sqrt{2 \pi n} \left(\frac{n}{e}\right)^{n}\right\}^{k}} = \sqrt{\frac{k}{\left(2 \pi n\right)^{k-1}}} k^{kn} = \Theta\left(\frac{k^{kn}}{n^{(k-1)/2}}\right) \nonumber
\end{align}
many ways. The number of possible $a_{i,i}$-regular graphs on $\mathcal{C}_{i}$, for each $i = 1, \ldots, k$, equals, by \eqref{regular_graph_count},
\begin{align}
\Theta\left(\frac{(n a_{i,i})!}{(n a_{i,i}/2)! 2^{n a_{i,i}/2}(a_{i,i}!)^{n}}\right) &= \Theta\left(\frac{\sqrt{2 \pi (a_{i,i} n)} \left(\frac{a_{i,i} n}{e}\right)^{a_{i,i} n}}{\sqrt{\pi (a_{i,i} n)} \left(\frac{a_{i,i} n}{2 e}\right)^{a_{i,i} n/2} 2^{n a_{i,i}/2} \left\{\sqrt{2 \pi a_{i,i}} \left(\frac{a_{i,i}}{e}\right)^{a_{i,i}}\right\}^{n}}\right) \nonumber\\
&= \Theta\left(\frac{a_{i,i}^{a_{i,i} n} n^{a_{i,i} n} e^{-a_{i,i} n}}{a_{i,i}^{a_{i,i} n/2} n^{a_{i,i} n/2} 2^{-a_{i,i}n/2} e^{-a_{i,i}n/2} 2^{a_{i,i} n/2} \left(2 \pi a_{i,i}\right)^{n/2} a_{i,i}^{a_{i,i} n} e^{-a_{i,i} n}}\right) \nonumber\\
&= \Theta\left(\frac{n^{a_{i,i} n/2} e^{a_{i,i}n/2} }{a_{i,i}^{(a_{i,i}+1)n/2} (2\pi)^{n/2}}\right). \nonumber
\end{align}
Similarly, by \eqref{bipartite_regular_graph_count}, the number of possible $a_{i,j}$-bipartite-regular graphs across clusters $\mathcal{C}_{i}$ and $\mathcal{C}_{j}$, for each $i \neq j$ and $i, j \in \{1, \ldots, k\}$, is given by
\begin{align}
\Theta\left(\frac{(a_{i,j} n)!}{(a_{i,j}!)^{2n}}\right) &= \Theta\left(\frac{\sqrt{2 \pi (a_{i,j} n)} \left(\frac{a_{i,j} n}{e}\right)^{a_{i,j} n}}{\left\{\sqrt{2 \pi a_{i,j}} \left(\frac{a_{i,j}}{e}\right)^{a_{i,j}}\right\}^{2n}}\right) = \Theta\left(\frac{\sqrt{n} a_{i,j}^{a_{i,j} n} n^{a_{i,j} n} e^{-a_{i,j} n}}{\left(2 \pi a_{i,j}\right)^{n} a_{i,j}^{2 a_{i,j} n} e^{-2 a_{i,j} n}}\right) = \Theta\left(\frac{\sqrt{n} n^{a_{i,j} n} e^{a_{i,j} n}}{(2 \pi)^{n} a_{i,j}^{(a_{i,j}+1)n}}\right). \nonumber
\end{align}

Therefore, combining these estimates, the total number of possible realizations of $\mathcal{G}^{n}_{A}$ on a given set of $k n$ labeled vertices becomes
\begin{align}\label{RSBM_model_count_estimate}
& \Theta\left(\frac{k^{kn}}{n^{(k-1)/2}} \prod_{i=1}^{k} \frac{n^{a_{i,i} n/2} e^{a_{i,i}n/2}}{a_{i,i}^{(a_{i,i}+1)n/2} (2\pi)^{n/2}} \prod_{i < j} \frac{\sqrt{n} n^{a_{i,j} n} e^{a_{i,j} n}}{(2 \pi)^{n} a_{i,j}^{(a_{i,j}+1)n}}\right) \nonumber\\
&= \Theta\left(\frac{k^{k n} n^{\sum_{i=1}^{k} n/2 \sum_{j=1}^{k} a_{i,j}} e^{\sum_{i=1}^{k} n/2 \sum_{j=1}^{k} a_{i,j}} n^{k(k-1)/4}}{n^{(k-1)/2} \prod_{i, j} a_{i,j}^{(a_{i,j}+1)n/2} (2\pi)^{nk/2} (2 \pi)^{k (k-1) n/2}}\right) = \Theta\left(\frac{k^{k n} n^{k n d/2} e^{k n d/2} n^{(k^{2} - 3k + 2)/4}}{\prod_{i, j} a_{i,j}^{(a_{i,j}+1)n/2} (2\pi)^{nk^{2}/2}}\right).
\end{align}
On the other hand, from \eqref{regular_graph_count}, the number of $d$-regular graphs on $kn$ vertices is 
\begin{align}\label{uniform_regular_count_estimate}
\Theta\left(\frac{(k n d)!}{(k n d/2)! 2^{k n d/2} (d!)^{k n}}\right) &= \Theta\left(\frac{\sqrt{2 \pi (k n d)} \left(\frac{k n d}{e}\right)^{k n d}}{\sqrt{2 \pi (k n d/2)} \left(\frac{k n d}{2 e}\right)^{k n d/2} 2^{k n d/2} \left(\sqrt{2 \pi d} \left(\frac{d}{e}\right)^{d}\right)^{k n}}\right) 
= \Theta\left(\frac{(k n d)^{k n d/2} e^{k n d/2}}{(2 \pi)^{k n/2} d^{(d+1/2)k n}}\right).
\end{align}

We recall here the well-known weighted geometric mean -- harmonic mean inequality. Given $m \in \mathbb{N}$ and positive reals $x_{i}$ and $\alpha_{i} $ for all $i \in \{1, \ldots, m\}$ with $\sum_{i=1}^{m} \alpha_{i} = 1$, we have
\begin{equation}\label{gm_hm}
\prod_{i=1}^{m} x_{i}^{\alpha_{i}} \geqslant \left(\sum_{i=1}^{m} \frac{\alpha_{i}}{x_{i}}\right)^{-1}.
\end{equation}
Setting $x_{j} = a_{i,j}$ and $\alpha_{j} = \frac{a_{i,j}}{d}$ for all $j \in \{1, \ldots, k\}$, from \eqref{gm_hm}, we get
\begin{equation}
\prod_{j=1}^{k} a_{i,j}^{a_{i,j}/d} \geqslant \frac{d}{k} \quad \implies \quad \prod_{j=1}^{k} a_{i,j}^{a_{i,j}} \geqslant \frac{d^{d}}{k^{d}} \nonumber
\end{equation}
for each $i = 1, \ldots, k$. This in turn yields
\begin{equation}\label{gm_hm_combined}
\left\{\prod_{i, j} a_{i,j}^{a_{i,j}}\right\}^{n/2} \geqslant \frac{d^{k n d/2}}{k^{k n d/2}}.
\end{equation}
We also observe that given positive integers $x_{1}, \ldots, x_{k}$ for any $k \in \mathbb{N}$, the following inequality holds:
\begin{equation}\label{product_sum}
k \prod_{i=1}^{k} x_{i} \geqslant \sum_{i=1}^{k} x_{i}.
\end{equation}

From \eqref{RSBM_model_count_estimate}, \eqref{uniform_regular_count_estimate}, \eqref{gm_hm_combined} and \eqref{product_sum}, we see that
\begin{align}\label{eq_1}
\mu^{kn}_{d}\left(\mathcal{S}^{n}_{A}\right) &\leqslant \Theta\left(\frac{k^{k n} n^{k n d/2} e^{k n d/2} n^{(k^{2} - 3k + 2)/4}}{\prod_{i, j} a_{i,j}^{(a_{i,j}+1)n/2} (2\pi)^{nk^{2}/2}} \cdot \frac{(2 \pi)^{k n/2} d^{(d+1/2)k n}}{(k n d)^{k n d/2} e^{k n d/2}}\right) \nonumber\\
&\leqslant \Theta\left(\frac{k^{k n} n^{(k^{2}-3k+2)/4} d^{k n d/2} d^{k n/2}}{\left\{\prod_{i, j} a_{i,j}^{a_{i,j}}\right\}^{n/2} \left\{\prod_{i, j \in [k]} a_{i,j}\right\}^{n/2} (2 \pi)^{k(k-1)n/2} k^{k n d/2}}\right) \nonumber\\
&\leqslant \Theta\left(\frac{k^{k n} n^{(k^{2}-3k+2)/4} d^{k n/2}}{\left\{\prod_{i, j} a_{i,j}\right\}^{n/2} (2 \pi)^{k(k-1)n/2}}\right)
= \Theta\left(\left\{\frac{k^{3}}{(2 \pi)^{k-1}}\right\}^{k n/2} \cdot n^{(k^{2}-3k+2)/4}\right).
\end{align} 
The ratio $\frac{k^{3}}{(2 \pi)^{k-1}}$ is strictly less than $1$ for all $k$ that satisfy $\frac{k-1}{\log k} > \frac{3}{\log (2\pi)}$. Now, the function $f(x) = \frac{x-1}{\log x}$ is strictly increasing in $x$ for all $x \geqslant 2$, and $f(3) = \frac{2}{\log 3} > \frac{3}{\log (2\pi)}$. This shows that the ratio $\frac{k^{3}}{(2 \pi)^{k-1}}$ is strictly less than $1$ for all $k \geqslant 3$, thus showing that the bound in \eqref{eq_1} is $o(1)$.


\section{Almost sure uniqueness of clusters in RSBM}\label{sec:unique}
This section is devoted to establishing fairly general sufficient conditions under which the communities involved in the regular stochastic block model are unique asymptotically almost surely as the size of each community approaches infinity. This is what constitutes our second main result. 
\begin{theorem}\label{main_2}
Suppose the model described in \S\ref{model_defn} satisfies the following conditions:
\begin{enumerate}
\item \label{intra_degrees_same} There exists a positive constant $C$ such that $C a_{i,i} > B_{i}$ for each $i \in [k]$, where $B_{i} = \max\left\{a_{i,\ell}: \ell \in \{1, \ldots, k\} \setminus \{i\}\right\}$.
\item When not all intra-cluster degrees are equal, for $i, j \in \{1, \ldots, k\}$ with $a_{i,i} < a_{j,j}$, there exist constants $\delta_{i,j} \in \left(0, \frac{1}{4}\right)$, independent of all entries of $A$, such that  
\begin{equation}\label{not_all_intra_degrees_same_cond'}
a_{i,i} \geqslant \left(\frac{1}{2} + 2\delta_{i,j}\right)a_{j,j}.
\end{equation}
Moreover, there exist constants $\epsilon_{i,j} \in \left(0, \frac{1}{4}\right)$, independent of the entries of $A$, such that for all $i, j \in \{1, \ldots, k\}$ with $a_{i,i} > a_{j,j}$ and $B_{j} > a_{i,i} - a_{j,j}$,
\begin{equation}\label{not_all_intra_degrees_same_cond}
\left(\frac{1}{2} - 2\epsilon_{i,j}\right)a_{i,i} \geqslant B_{j}.
\end{equation}  
\end{enumerate}
Then, for all sufficiently large values of the entries of the matrix $A$, the clusters $\mathcal{C}_{i}$, $i = 1, \ldots, k$, are a.a.s.\ unique as $n \rightarrow \infty$ while the matrix $A$ stays fixed. 
\end{theorem}

From the discussion in \S\ref{configuration:subsec}, it suffices to establish Theorem~\ref{main_2} on the random multigraph in which each of the intra-cluster regular graphs and inter-cluster bipartite-regular graphs is generated via the configuration model. Fix \emph{any} non-negative $\alpha_{1}, \ldots, \alpha_{k}$ such that $\sum_{i=1}^{k} \alpha_{i} = 1$, and subsets $C_{i}$ of $\mathcal{C}_{i}$ such that $|C_{i}| = \alpha_{i} n$ for $i = 1, \ldots, k$. Let $D = \bigcup_{i=1}^{k} C_{i}$. To prove Theorem~\ref{main_2}, it is enough to establish Proposition~\ref{C_not_cluster}. We note here that although   Proposition~\ref{C_not_cluster} is stated for $a_{1,1}$, its proof will be analogous if we replace $a_{1,1}$ by any $a_{i,i}$ for $i \in \{2, \ldots, k\}$.

\begin{prop}\label{C_not_cluster}
Assume that the hypotheses of Theorem~\ref{main_2} hold, and that the entries of $A$ are sufficiently large. Suppose there exist at least two distinct $i, j \in \{1, \ldots, k\}$ such that $\alpha_{i}$ and $\alpha_{j}$ are strictly positive. Then a.a.s.\ the following cannot be true simultaneously:
\renewcommand{\theenumi}{\roman{enumi}}
\begin{enumerate}
\item the subgraph $\mathcal{G}|_{D}$ is $a_{1,1}$-regular;
\item there exists a partition of $V(G) \setminus D$ into subsets $D_{2}, \ldots, D_{k}$, each of size $n$, such that $\mathcal{G}|_{D_{i}}$ is $a_{i,i}$-regular, $\mathcal{G}|_{D_{i}, D_{j}}$ is $a_{i,j}$-bipartite-regular and $\mathcal{G}|_{D, D_{i}}$ is $a_{1,i}$-bipartite-regular for all distinct $i, j \in \{2, \ldots, k\}$.
\end{enumerate}
\end{prop}

In order to prove Proposition~\ref{C_not_cluster}, we start with the assumption that $\mathcal{G}|_{D}$ is $a_{1,1}$-regular. The proof requires consideration of a few different cases depending on the values of the $\alpha_{i}$'s, and these are addressed in Lemma~\ref{simpler}, \S\ref{subcase_1}, \S\ref{subcase_2} and \S\ref{subcase_3}.

\begin{lemma}\label{simpler}
If $a_{i,i} > a_{1,1}$ and $\alpha_{i} > \frac{B_{1}}{a_{i,i} - a_{1,1} + B_{1}}$, then the conclusion of Proposition~\ref{C_not_cluster} holds.
\end{lemma}   

\begin{proof}
Assume that $\alpha_{i} > 0$ for some $i$ such that $a_{i,i} > a_{1,1}$, and that Proposition~\ref{C_not_cluster} does \emph{not} hold. As $\mathcal{G}|_{D}$ is $a_{1,1}$-regular, we have $\deg_{\mathcal{C}_{i} \setminus C_{i}}(v) \geqslant (a_{i,i} - a_{1,1})$ for every $v \in C_{i}$, yielding $\deg(C_{i}, \mathcal{C}_{i} \setminus C_{i}) \geqslant (a_{i,i} - a_{1,1}) \alpha_{i} n$. On the other hand, each $u$ in $\mathcal{C}_{i} \setminus C_{i}$ belongs to precisely one of the remaining clusters $D_{2}, \ldots, D_{k}$. If $u \in D_{j}$, then $\deg_{C_{i}}(u) \leqslant \deg_{D}(u) = a_{1,j}$. Thus $\deg(C_{i}, \mathcal{C}_{i} \setminus C_{i}) \leqslant B_{1} (1-\alpha_{i}) n$. These two inequalities together yield
\begin{equation}\label{better_bound_on_alpha_{i}}
(a_{i,i} - a_{1,1}) \alpha_{i} \leqslant B_{1} (1-\alpha_{i}) \quad \implies \quad \alpha_{i} \leqslant \frac{B_{1}}{a_{i,i} - a_{1,1} + B_{1}}, \nonumber
\end{equation}
thus completing the proof.
\end{proof}

From here onward, we only consider those $\ell \in \{1, \ldots, k\}$ such that $\alpha_{\ell} > 0$, without mentioning so every time. We shall let $i$ denote that index in $\{1, \ldots, k\}$ (if this is not unique, we choose any such $i$ and fix it) for which $\alpha_{i} \geqslant \alpha_{\ell}$ for all $\ell \in \{1, \ldots, k\} \setminus \{i\}$. Note that this guarantees, by the pigeon hole principle, that $\alpha_{i} \geqslant \frac{1}{k}$. 

Under the assumption that $\mathcal{G}|_{D}$ is $a_{1,1}$-regular, we have $\sum_{\ell=1}^{k} \deg_{C_{j}}(v) = a_{1,1}$ for every $v \in C_{i}$. On the other hand, $\deg_{C_{i}}(v) + \deg_{\mathcal{C}_{i} \setminus C_{i}}(v) = a_{i,i}$ for each $v \in C_{i}$. These together imply
\begin{equation}\label{degree_identity}
\sum_{\ell \in \{1, \ldots, k\} \setminus \{i\}} \deg_{C_{\ell}}(v) = a_{1,1} - a_{i,i} + \deg_{\mathcal{C}_{i} \setminus C_{i}}(v).
\end{equation}

To prove Proposition\ref{C_not_cluster}, we first condition on the $\sigma$-field $\mathcal{F}$ comprising the following information:
\renewcommand{\theenumi}{\roman{enumi}}
\begin{enumerate}
\item the vertex sets of $\mathcal{C}_{\ell}$ and $C_{\ell}$ for all $\ell = 1, \ldots, k$,
\item the subgraph $\mathcal{G}|_{\mathcal{C}_{i}}$ induced on $\mathcal{C}_{i}$.
\end{enumerate}
Given $\mathcal{F}$, we enumerate the vertices of $\mathcal{C}_{i}$ as $v_{1}, \ldots, v_{n}$ such that $C_{i} = \left\{v_{1}, \ldots, v_{\alpha_{i}n}\right\}$. From \eqref{degree_identity}, we set 
\begin{equation}\label{g_{s}_values}
g_{s} = a_{1,1} - a_{i,i} + \deg_{\mathcal{C}_{i} \setminus C_{i}}(v_{s})
\end{equation}
for all $s = 1, \ldots, \alpha_{i} n$. The random variables $g_{s}$ are measurable with respect to $\mathcal{F}$. The conditional probability of the event that $\mathcal{G}|_{D}$ is $a_{1,1}$-regular is bounded above by the conditional probability of the event 
\begin{equation}
A = \left\{\sum_{\ell \in \{1, \ldots, k\} \setminus \{i\}} \deg_{C_{\ell}}(v_{s}) = g_{s} \text{ for all } s = 1, \ldots, \alpha_{i} n\right\}. \nonumber
\end{equation}
We show that the probability of the event $A$ is $o(1)$ as $n \rightarrow \infty$. 

First, we express $A$ as the union of pairwise disjoint events. For $g \in \mathbb{N}$, let us define the following subset of ordered $(k-1)$-tuples of non-negative integers:
\begin{equation}
\mathcal{S}_{g} = \left\{\left(m_{1}, \ldots, m_{i-1}, m_{i+1}, \ldots, m_{k}\right): 0 \leqslant m_{\ell} \leqslant a_{i,\ell} \text{ for all } \ell \in \{1, \ldots, k\} \setminus \{i\}, \sum_{\ell \in \{1, \ldots, k\} \setminus \{i\}} m_{\ell} = g\right\}. \nonumber
\end{equation} 
Then $A$ can be written as the union of the events
\begin{equation}
A\left(m^{(s)}: s = 1, \ldots, \alpha_{i} n\right) = \left\{\deg_{C_{\ell}}(v_{s}) = m^{(s)}_{\ell} \text{ for all } \ell \in \{1, \ldots, k\} \setminus \{i\} \text{ and } s = 1, \ldots, \alpha_{i} n\right\} \nonumber
\end{equation}
where $m^{(s)} = \left(m^{(s)}_{\ell}: \ell \in \{1, \ldots, k\} \setminus \{i\}\right)$ belongs to $\mathcal{S}_{g_{s}}$ for all $s = 1, \ldots, \alpha_{i} n$. Note, from the mutual independence of the subgraphs $\mathcal{G}|_{\mathcal{C}_{i}, \mathcal{C}_{\ell}}$ for $\ell \in \{1, \ldots, k\} \setminus \{i\}$, that 
\begin{align}\label{independence}
\Prob\left[A\left(m^{(s)}: s = 1, \ldots, \alpha_{i} n\right)\Big|\mathcal{F}\right] &= \prod_{\ell \in \{1, \ldots, k\} \setminus \{i\}} \Prob\left[\deg_{C_{\ell}}(v_{s}) = m^{(s)}_{\ell} \text{ for all } s = 1, \ldots, \alpha_{i} n\Big|\mathcal{F}\right] \nonumber\\
&\leqslant \Prob\left[\deg_{C_{\ell}}(v_{s}) = m^{(s)}_{\ell} \text{ for all } s = 1, \ldots, \alpha_{i} n\Big|\mathcal{F}\right]
\end{align}
for each $\ell \in \{1, \ldots, k\} \setminus \{i\}$. The goal now is to fix \emph{any} $m^{(s)} \in \mathcal{S}_{g_{s}}$ for each $s$ and establish that the probability of the event $\left\{\deg_{C_{\ell}}(v_{s}) = m^{(s)}_{\ell} \text{ for all } s = 1, \ldots, \alpha_{i} n\right\}$ for at least one $\ell \in \{1, \ldots, k\} \setminus \{i\}$ is $o(n^{-1})$ as $n \rightarrow \infty$. 


Since $m^{(s)} \in \mathcal{S}_{g_{s}}$ for each $s$, from \eqref{g_{s}_values}, we have
\begin{equation}\label{G_value}
\sum_{\ell \in \{1, \ldots, k\} \setminus \{i\}} \sum_{s=1}^{\alpha_{i} n} m^{(s)}_{\ell} = \sum_{s=1}^{\alpha_{i} n} \sum_{\ell \in \{1, \ldots, k\} \setminus \{i\}} m^{(s)}_{\ell} = \sum_{s=1}^{\alpha_{i} n} g_{s} = \deg(C_{i}, \mathcal{C}_{i} \setminus C_{i}) + (a_{1,1} - a_{i,i})\alpha_{i} n. 
\end{equation}
For $G$ uniformly randomly chosen from $\mathcal{R}^{n}_{d}$, [\cite{10}, Theorem 1.1] showed that $\gamma \geqslant 1 - \frac{2}{\sqrt{d}}$ a.a.s.\ as $n \rightarrow \infty$, where $\gamma$ is the spectral gap for the adjacency matrix of $G$. Given a $d$-regular graph $G$ on $n$ vertices and a subset $S$ of $V(G)$ with $|S| \leqslant \frac{n}{2}$, [\cite{06}, Theorem 13.14] (see also \cite{11}, \cite{12}, and [\cite{05}, Theorem 6]) established that
\begin{equation}
\frac{\gamma}{2} \leqslant \frac{\deg(S, V(G) \setminus S)}{d|S|}. \nonumber
\end{equation}
Combining these, we get 
\begin{equation}\label{C_{i}_mathcal{C}_{i}_deg}
\deg\left(\mathcal{C}_{i} \setminus C_{i}, C_{i}\right) \geqslant \min\{\alpha_{i}, 1 - \alpha_{i}\} \left(\frac{1}{2} - \frac{1}{\sqrt{a_{i,i}}}\right) a_{i,i} n.
\end{equation}
From \eqref{G_value} and \eqref{C_{i}_mathcal{C}_{i}_deg}, we get
\begin{equation}
\sum_{\ell \in \{1, \ldots, k\} \setminus \{i\}} \sum_{s=1}^{\alpha_{i} n} m^{(s)}_{\ell} \geqslant \min\{\alpha_{i}, 1 - \alpha_{i}\} \left(\frac{1}{2} - \frac{1}{\sqrt{a_{i,i}}}\right) a_{i,i} n + (a_{1,1} - a_{i,i})\alpha_{i} n. \nonumber
\end{equation}
By the pigeon-hole principle, there exists at least one $j \in \{1, \ldots, k\} \setminus \{i\}$ such that
\begin{equation}\label{choice_of_j}
\sum_{s=1}^{\alpha_{i} n} m^{(s)}_{j} \geqslant \frac{1}{k-1}\left\{\min\{\alpha_{i}, 1 - \alpha_{i}\} \left(\frac{1}{2} - \frac{1}{\sqrt{a_{i,i}}}\right) a_{i,i} n + (a_{1,1} - a_{i,i})\alpha_{i} n\right\}.
\end{equation}
For the rest of the proof, we fix such a $j$, and establish the following lemma:
\begin{lemma}\label{main_step_lemma}
Let $A_{j}$ denote the event that there exist $C_{i} = \{v_{1}, \ldots, v_{\alpha_{i} n}\} \subset \mathcal{C}_{i}$ and $C_{j} \subset \mathcal{C}_{j}$ such that $\deg_{C_{j}}(v_{s}) = m^{(s)}_{j}$ for all $s = 1, \ldots, \alpha_{i} n$. Then $\Prob[A_{j}] = o(n^{-1})$.
\end{lemma} 
The proof of this lemma is accomplished through the consideration of three different cases, in \S\ref{subcase_1}, \S\ref{subcase_2} and \S\ref{subcase_3}. 

\subsection{When $a_{1,1} \geqslant a_{i,i}$}\label{subcase_1} We note at the very outset that the analysis of \S\ref{subcase_1} is enough for the special and commonly studied situation where all intra-cluster degrees are the same. We set $G_{j} = \sum_{s=1}^{\alpha_{i} n} m^{(s)}_{j}$, so that from \eqref{choice_of_j}, we have, for all $a_{i,i}$ sufficiently large,
\begin{equation}\label{G_{j}_lower_bound_subcase_1}
\frac{G_{j}}{n} \geqslant \frac{\min\{\alpha_{i}, 1 - \alpha_{i}\} a_{i,i}}{4(k-1)}.
\end{equation}

We refer the reader to [\cite{09}, Lemma 2] for the following inequality: 
\begin{multline}\label{optimize_probab}
\Prob\left[\deg_{C_{j}}(v_{s}) = m^{(s)}_{j} \text{ for all } s = 1, \ldots, \alpha_{i} n\Big|\mathcal{F}\right] \\ \leqslant \Prob\left[\deg_{C_{j}}(v_{s}) \in \{\eta_{j}, \eta_{j}+1\} \text{ for all } s = 1, \ldots, \alpha_{i} n, \sum_{s=1}^{\alpha_{i} n} \deg_{C_{j}}(v_{s}) = \sum_{s=1}^{\alpha_{i} n} m^{(s)}_{j}\Big|\mathcal{F}\right], 
\end{multline}
where $\eta_{j} = \left\lfloor \frac{G_{j}}{\alpha_{i} n}\right\rfloor$. Notice that we have the trivial bound $G_{j} = \deg(C_{i}, C_{j}) \leqslant \deg(C_{j}, \mathcal{C}_{i}) = \alpha_{j} a_{i,j} n$, so that 
\begin{equation}\label{eta_{j}_upper_bound}
\eta_{j} \leqslant \frac{\alpha_{j} a_{i,j}}{\alpha_{i}}.
\end{equation}
We invoke the configuration model discussed in \S\ref{configuration:subsec}, and outline, in the next paragraph, some foundational aspects of the argument that resemble [\cite{09}, Lemma 5]. Let $\xi_{a_{i,j}(s-1)+1}, \ldots, \xi_{a_{i,j}s}$ denote the half-edges of type $\{i, j\}$ emanating from vertex $v_{s}$, for each $s = 1, \ldots, \alpha_{i} n$. Let $B_{t}$ denote the indicator random variable of the event that $\xi_{t}$ is matched with a half-edge of type $\{i, j\}$ emanating from $C_{j}$. Conditioned on $B_{1}, \ldots, B_{t}$, the random variable $B_{t+1}$ is Bernoulli with probability 
\begin{equation}
\hat{p}_{t} = \frac{\alpha_{j} a_{i,j} n - \sum_{t' \leqslant t} B_{t'}}{a_{i,j} n - t}.\nonumber
\end{equation}
For all $1 \leqslant t \leqslant \alpha_{i}n$, we see that $\left|\hat{p}_{t} - \hat{p}_{t-1}\right| \leqslant O(n^{-1})$, so that for all $s = 1, \ldots, \alpha_{i} n$, there exists $p_{s} \in (0,1)$ such that
\begin{equation}\label{bin_approx}
\left|\left|\deg_{C_{j}}(v_{s})\Big|\mathcal{G}_{s-1}, \bin\left(a_{i,j}, p_{s}\right)\right|\right|_{\tv} = O\left(\frac{1}{n}\right),
\end{equation}
where $\mathcal{G}_{s}$ denotes the $\sigma$-field generated by $\left\{\xi_{t}, 1 \leqslant t \leqslant a_{i,j}s\right\}$, for each $s$. Given that each of $\deg_{C_{j}}(v_{1})$, $\ldots$, $\deg_{C_{j}}(v_{s-1})$ takes values in $\{\eta_{j}, \eta_{j}+1\}$, the number of half-edges emanating from $C_{j}$ that have not yet been matched is at least $\alpha_{j} a_{i,j} n - (\eta_{j}+1)(s-1)$ and at most $\alpha_{j} a_{i,j} n - \eta_{j}(s-1)$. The number of half-edges from $\mathcal{C}_{j}$ that are left to be matched is $a_{i,j}(n - s + 1)$. Thus
\begin{equation}\label{range_p_s}
\frac{\alpha_{j} a_{i,j} n - (\eta_{j}+1)(s-1)}{a_{i,j}(n-s+1)} \leqslant p_{s} \leqslant \frac{\alpha_{j} a_{i,j} n - \eta_{j}(s-1)}{a_{i,j}(n-s+1)}.
\end{equation}

We shall now consider three different ranges of values of $\alpha_{j}$, where $j$ is as chosen by \eqref{choice_of_j}. First, consider
\begin{equation}\label{range_1}
\frac{c}{a_{i,j}} < \alpha_{j} \leqslant \frac{1}{2},
\end{equation}
where 
\begin{equation}\label{cond_c_subcase_1}
\log_{2}c > \max\left\{16k + 1, \frac{2(k-1)C}{4^{8k-1}}\right\},
\end{equation}
where $C$ is as in \ref{intra_degrees_same} of Theorem~\ref{main_2}. We first consider the case of $\eta_{j} \neq 0$. For each $s = 1, \ldots, \left\lfloor\frac{\alpha_{i} n}{4}\right\rfloor$, from \eqref{eta_{j}_upper_bound} and \eqref{range_p_s}, we have
\begin{equation}\label{p_s_lower_bound}
p_{s} \geqslant \frac{\alpha_{j} a_{i,j} n - 2 \eta_{j} \cdot \frac{\alpha_{i} n}{4}}{a_{i,j} n} \geqslant \frac{\alpha_{j} a_{i,j} n - 2 \cdot \frac{\alpha_{j} a_{i,j}}{\alpha_{i}} \cdot \frac{\alpha_{i} n}{4}}{a_{i,j} n} = \frac{\alpha_{j}}{2} > \frac{c}{2 a_{i,j}}.
\end{equation}
Using \eqref{bin_approx}, \eqref{p_s_lower_bound}, the fact that the mode of the $\bin(n, p)$ distribution is $\lfloor (n+1)p \rfloor$, and the same argument as in [\cite{09}, Lemma 4], we conclude that, for all $s = 1, \ldots, \left\lfloor\frac{\alpha_{i} n}{4}\right\rfloor$, 
\begin{align}\label{each_deg_probab_lower_bound_subcase_1}
\Prob\left[\deg_{C_{j}}(v_{s}) \in \{\eta_{j}, \eta_{j}+1\}\Big|\mathcal{G}_{s-1}\right] &\leqslant \Prob\left[\bin\left(a_{i,j}, p_{s}\right) = \lfloor (a_{i,j} + 1) p_{s}\rfloor\right] + O\left(\frac{1}{n}\right) \nonumber\\
&\leqslant O\left(\sqrt{\frac{1}{a_{i,j} p_{s}}}\right) \leqslant O\left(\sqrt{\frac{2}{c}}\right).
\end{align}
We now use union bounds and Stirling's approximation to bound above $\Prob[A_{j}]$ by
\begin{align}
{n \choose \alpha_{i}n} {n \choose \alpha_{j}n} \left\{O\left(\sqrt{\frac{2}{c}}\right)\right\}^{\frac{\alpha_{i} n}{4}} &\sim O\left(\frac{2^{H(\alpha_{i})n + H(\alpha_{j})n}}{\sqrt{\alpha_{i} \alpha_{j} (1-\alpha_{i}) (1-\alpha_{j})} n} \left(\frac{2}{c}\right)^{\frac{\alpha_{i} n}{8}}\right) \nonumber\\
&\leqslant O\left(\frac{2^{2n}}{\sqrt{\alpha_{i} \alpha_{j} (1-\alpha_{i}) (1-\alpha_{j})} n} \left(\frac{2}{c}\right)^{\frac{n}{8k}}\right) \nonumber
\end{align}
which is $o(n^{-1})$ by \eqref{cond_c_subcase_1}. 

Now we consider the case of $\eta_{j} = 0$. We enumerate the vertices of $C_{j}$ so that $\deg_{C_{j}}(v_{s}) = 1$ for $s = 1, \ldots, G_{j}$ and $\deg_{C_{j}}(v_{s}) = 0$ for all $G_{j}+1 \leqslant s \leqslant \alpha_{i}n$. Inspired by the lower bound in \eqref{range_p_s}, set
\begin{equation}
f(s-1) = \frac{\alpha_{j} a_{i,j} n - s + 1}{a_{i,j}(n-s+1)}, \text{ for all } s \in [G_{j}]. \nonumber
\end{equation} 
The function $f$ being strictly increasing for $\alpha_{j}$ as in \eqref{range_1}, a uniform lower bound on $p_{s}$ for all $s = 1, \ldots, G_{j}$ is $f(0) = \alpha_{j} > \frac{c}{a_{i,j}}$. By similar computations as used in deriving in \eqref{each_deg_probab_lower_bound_subcase_1}, we get
\begin{equation}\label{each_deg_probab_lower_bound_subcase_1_eta_0}
\Prob\left[\deg_{C_{j}}(v_{s}) = 1\Big|\mathcal{G}_{s-1}\right] \leqslant O\left(\sqrt{\frac{1}{c}}\right)
\end{equation}
for all $s = 1, \ldots, G_{j}$. Note that $\alpha_{j} \leqslant \alpha_{i}$ by our choice of $i$, and $\alpha_{i} + \alpha_{j} \leqslant 1$ implies that $\alpha_{j} \leqslant 1 - \alpha_{i}$. Hence, from \eqref{G_{j}_lower_bound_subcase_1}, \eqref{range_1} and \eqref{cond_c_subcase_1} we conclude that
\begin{equation}\label{G_{j}_lower_bound_subcase_1_subsub_1}
\frac{G_{j}}{n} \geqslant \frac{\alpha_{j} a_{i,i}}{4(k-1)} > \frac{c a_{i,i}}{4 (k-1) a_{i,j}} > \frac{2 \cdot 4^{8k-1}}{(k-1)C}.
\end{equation}
In this case, the upper bound on $\Prob[A_{j}]$ is given by
\begin{align}
{n \choose \alpha_{i}n} {n \choose \alpha_{j}n} \left\{O\left(\sqrt{\frac{1}{c}}\right)\right\}^{G_{j}} &\leqslant O\left(\frac{2^{2n}}{\sqrt{\alpha_{i} \alpha_{j} (1-\alpha_{i}) (1-\alpha_{j})} n} \left(\frac{1}{c}\right)^{\frac{4^{8k-1} n}{(k-1)C}}\right) \nonumber
\end{align}
which is $o(n^{-1})$ due to \eqref{cond_c_subcase_1}.

Next, we consider the following range of values of $\alpha_{j}$:
\begin{equation}\label{range_2}
\frac{1}{a_{i,j}^{2}} < \alpha_{j} \leqslant \frac{c}{a_{i,j}},
\end{equation}
where $c$ is as in \eqref{cond_c_subcase_1}. At the very outset of this case, we note that, if $\alpha_{i} \leqslant \frac{1}{2}$, we must have
\begin{align}
c \geqslant \alpha_{j} a_{i,j} \geqslant \frac{G_{j}}{n} \geqslant \frac{\alpha_{i} a_{i,i}}{4(k-1)}, \nonumber
\end{align}
implying that $\alpha_{i} \leqslant \frac{4(k-1)c}{a_{i,i}}$. For all $a_{i,i}$ sufficiently large, this upper bound is smaller than $\frac{1}{k}$, giving us a contradiction. 

\begin{rem}\label{no_need_1}
The above reasoning shows that for the range $\alpha_{j} \leqslant \frac{c}{a_{i,j}}$, we need not consider $\alpha_{i} \leqslant \frac{1}{2}$. 
\end{rem} 

When $\alpha_{i} > \frac{1}{2}$, we have, by similar reasoning as above,
\begin{equation}\label{lower_bound_a_{i,i}_subcase_1_subsub_2}
c \geqslant \frac{(1-\alpha_{i}) a_{i,i}}{4(k-1)} \quad \implies \quad \alpha_{i} \geqslant 1 - \frac{4 c (k-1)}{a_{i,i}}.
\end{equation}
For all $a_{i,i}$ sufficiently large, this yields:
\begin{equation}\label{ineq_seq}
\alpha_{j} < 1 - \alpha_{i} \leqslant \frac{4c (k-1)}{a_{i,i}} < \frac{1}{2} < 1 - \frac{4 c (k-1)}{a_{i,i}} \leqslant \alpha_{i} < 1 - \alpha_{j},
\end{equation}
and by the concave nature of the entropy function and its symmetry around $\frac{1}{2}$, we conclude that
\begin{equation}\label{entropy_bound_subcase_1_subsub_2}
H(\alpha_{i}) + H(\alpha_{j}) \leqslant 2 H\left(\frac{4 c (k-1)}{a_{i,i}}\right).
\end{equation}

We first address the case of $\eta_{j} \neq 0$. From \eqref{range_p_s} and \eqref{range_2}, for all $s = 1, \ldots, \left\lfloor\frac{\alpha_{i} n}{2}\right\rfloor$, we get
\begin{equation}\label{p_s_upper_subcase_1}
p_{s} \leqslant \frac{\alpha_{j} a_{i,j} n}{a_{i,j}\left(n - \frac{\alpha_{i}n}{2}\right)} < \frac{2c}{a_{i,j}}, 
\end{equation} 
so that $\deg_{C_{j}}(v_{s})$, conditioned on $\mathcal{G}_{s-1}$, is stochastically dominated by $\bin\left(a_{i,j}, \frac{2c}{a_{i,j}}\right)$, which in turn can be approximated by the $\poi(2c)$ distribution. Thus 
\begin{equation}\label{deg_upper_bound_subcase_1_subsub_2}
\Prob\left[\deg_{C_{j}}(v_{s}) \in \left\{\eta_{j}, \eta_{j+1}\right\}\Big|\mathcal{G}_{s-1}\right] \leqslant \Prob\left[\poi(2c) \geqslant 1\right] = \gamma
\end{equation}
where $\gamma$ is a constant that depends only on $c$. Using \eqref{entropy_bound_subcase_1_subsub_2} and \eqref{deg_upper_bound_subcase_1_subsub_2}, and $\frac{\alpha_{i} n}{2} > \frac{n}{4}$, we get the following upper bound on $\frac{1}{n} \log_{2} \Prob[A_{j}]$:
\begin{equation}\label{log_bound_subcase_1_subsub_2}
-\frac{\log_{2}\left(\alpha_{i} \alpha_{j} (1-\alpha_{i}) (1-\alpha_{j})\right)}{2n} - \frac{\log_{2} n}{n} + 2 H\left(\frac{4 c (k-1)}{a_{i,i}}\right) + \frac{\log_{2} \gamma}{4}.
\end{equation}
The first two terms approach $0$ as $n \rightarrow \infty$. The last term is a strictly negative constant, and as $a_{i,i}$ grows, the third term goes to $0$. Hence \eqref{log_bound_subcase_1_subsub_2} is strictly negative for all sufficiently large $a_{i,i}$, as $n \rightarrow \infty$. 

\begin{rem}\label{no_need_2}
Observe that, in the above argument, nowhere has the lower bound on $\alpha_{j}$ from \eqref{range_2} been used. This shows that as far as the case of $\eta_{j} \neq 0$ is concerned, our proof of Lemma~\ref{main_step_lemma} for the regime of \S\ref{subcase_1} ends here. In the rest of \S\ref{subcase_1}, we only consider $\eta_{j} = 0$.
\end{rem}

Now, we consider $\eta_{j} = 0$ and $\alpha_{j}$ in the range given by \eqref{range_2}. If $G_{j} \geqslant \frac{\alpha_{i}n}{2}$, then the same argument as above will be enough. 
\begin{rem}\label{no_need_3}
This shows that for all $\alpha_{j} \leqslant \frac{c}{a_{i,j}}$ and $\eta_{j} = 0$, as long as $G_{j} \geqslant \frac{\alpha_{i}n}{2}$, our proof of Lemma~\ref{main_step_lemma} is already complete. Henceforth, we only consider $\eta_{j} = 0$ and $G_{j} < \frac{\alpha_{i} n}{2}$.
\end{rem}

If $G_{j} < \frac{\alpha_{i} n}{2}$, then for each $s = 1, \ldots, G_{j}$, the bound in \eqref{p_s_upper_subcase_1} holds, and hence so does \eqref{deg_upper_bound_subcase_1_subsub_2}. Together with \eqref{G_{j}_lower_bound_subcase_1} and \eqref{ineq_seq}, this yields the following upper bound on $\frac{1}{n} \log_{2}\Prob[A_{j}]$:
\begin{equation}
-\frac{\log_{2}(\alpha_{i} \alpha_{j} (1-\alpha_{i}) (1-\alpha_{j}))}{2n} - \frac{\log_{2} n}{n} + 2 H(\alpha_{i}) + \frac{(1-\alpha_{i}) a_{i,i}}{4(k-1)} \log_{2} \gamma. \nonumber
\end{equation}
Again, the first two terms approach $0$ as $n \rightarrow \infty$. We focus on the last two terms. Using the lower bound on $\alpha_{j}$ from \eqref{range_2} and the fact that $x \log x > (1-x) \log (1-x)$ for all $x \in \left(\frac{1}{2}, 1\right)$, we get:
\begin{align}
2 H(\alpha_{i}) + \frac{(1-\alpha_{i}) a_{i,i}}{4(k-1)} \log_{2} \gamma &\leqslant -4 (1-\alpha_{i}) \log_{2} (1-\alpha_{i}) + \frac{(1-\alpha_{i}) a_{i,i} \log_{2} \gamma}{4(k-1)} \nonumber\\
&\leqslant 8 (1-\alpha_{i}) \log_{2} a_{i,j} + \frac{(1-\alpha_{i}) a_{i,i} \log_{2} \gamma}{4(k-1)} \nonumber\\
&\leqslant 8 (1-\alpha_{i}) \left\{\log_{2} C + \log_{2} a_{i,i}\right\} + \frac{(1-\alpha_{i}) a_{i,i} \log_{2} \gamma}{4(k-1)} \nonumber\\
&= (1-\alpha_{i}) \left\{8 \log_{2} C + 8 \log_{2} a_{i,i} + \frac{a_{i,i} \log_{2} \gamma}{4(k-1)}\right\}. \nonumber
\end{align} 
As $a_{i,i}$ grows to $\infty$ much faster than $\log_{2} a_{i,i}$, and the coefficient of $a_{i,i}$ is a strictly negative constant whereas that of $\log_{2} a_{i,i}$ is a positive one, hence this is strictly negative for all $a_{i,i}$ sufficiently large.

Finally, we consider 
\begin{equation}\label{range_3}
\alpha_{j} \leqslant \frac{1}{a_{i,j}^{2}},
\end{equation}
and by Remarks~\ref{no_need_1} and \ref{no_need_2}, we need only consider $\alpha_{i} > \frac{1}{2}$ and $\eta_{j} = 0$. From \eqref{range_p_s}, for all $s = 1, \ldots, \left\lfloor\frac{\alpha_{i} n}{2}\right\rfloor$, we have 
\begin{equation}\label{p_{s}_upper_bound_subcase_1_subsub_3}
p_{s} \leqslant \frac{\alpha_{j} a_{i,j} n}{a_{i,j} \left(n - \frac{\alpha_{i} n}{2}\right)} < 2 \alpha_{j} \leqslant 2(1-\alpha_{i}), \nonumber
\end{equation}
so that 
\begin{equation}\label{deg_upper_bound_subcase_1_subsub_3}
\Prob\left[\deg_{C_{j}}(v_{s}) \in \left\{\eta_{j}, \eta_{j}+1\right\}\Big|\mathcal{F}_{s-1}\right] \leqslant \Prob\left[\bin\left(a_{i,j}, p_{s}\right) \geqslant 1\right] + O(n^{-1}) \leqslant a_{i,j} p_{s} < 2 (1-\alpha_{i}) a_{i,j}.
\end{equation}
By Remark~\ref{no_need_3}, we need only consider $G_{j} < \frac{\alpha_{i} n}{2}$, so that \eqref{deg_upper_bound_subcase_1_subsub_3} holds for all $s = 1, \dots, G_{j}$. By \eqref{G_{j}_lower_bound_subcase_1} and \eqref{range_3}, we get:
\begin{equation}\label{1-alpha_{i}_upper_bound_subcase_1_subsub_3}
\frac{1}{a_{i,j}} \geqslant \alpha_{j} a_{i,j} \geqslant \frac{G_{j}}{n} \geqslant \frac{(1-\alpha_{i}) a_{i,i}}{4(k-1)} \quad \implies \quad 1-\alpha_{i}  \leqslant \frac{4(k-1)}{a_{i,i} a_{i,j}}.
\end{equation}
For any fixed positive integer $r > 2$, for all $a_{i,i}$ sufficiently large, by \ref{intra_degrees_same} of Theorem~\ref{main_2}, we have
\begin{align}
a_{i,i} > C^{\frac{1}{r-2}} \left\{4(k-1)\right\}^{\frac{r-1}{r-2}} 2^{\frac{r}{r-2}} &\implies a_{i,i}^{r-1} \geqslant \left\{4(k-1)\right\}^{r-1} 2^{r} C a_{i,i} \geqslant \left\{4(k-1)\right\}^{r-1} 2^{r} a_{i,j}, \nonumber
\end{align}
so that by \eqref{1-alpha_{i}_upper_bound_subcase_1_subsub_3} we get 
\begin{equation}\label{log_upper_bound_subcase_1_subsub_3}
\log_{2}\left\{2 (1-\alpha_{i}) a_{i,j}\right\} \leqslant \frac{\log_{2}(1-\alpha_{i})}{r}.
\end{equation}
By \eqref{G_{j}_lower_bound_subcase_1} and \eqref{deg_upper_bound_subcase_1_subsub_3}, we get the following upper bound on $\frac{1}{n} \log_{2} \Prob[A_{j}]$:
\begin{equation}
-\frac{\log_{2}(\alpha_{i} \alpha_{j} (1-\alpha_{i}) (1-\alpha_{j}))}{2n} - \frac{\log_{2} n}{n} + 2 H(\alpha_{i}) + \frac{(1-\alpha_{i}) a_{i,i}}{4(k-1)} \log_{2} \left\{2(1-\alpha_{i})a_{i,j}\right\}. \nonumber
\end{equation}
Again, it suffices to focus on the last two terms, and by \eqref{log_upper_bound_subcase_1_subsub_3}, we get the following bound:
\begin{align}
2 H(\alpha_{i}) + \frac{(1-\alpha_{i}) a_{i,i}}{4(k-1)} \log_{2} \left\{2(1-\alpha_{i})a_{i,j}\right\} &\leqslant -4(1-\alpha_{i}) \log_{2}(1-\alpha_{i}) + \frac{a_{i,i}}{4(k-1)} (1-\alpha_{i}) \frac{\log_{2}(1-\alpha_{i})}{r} \nonumber\\
&= (1-\alpha_{i}) \log_{2}(1-\alpha_{i})\left\{-4 + \frac{a_{i,i}}{4(k-1)r}\right\}, \nonumber
\end{align}
which is strictly negative for all $a_{i,i}$ sufficiently large. This brings us to the end of \S\ref{subcase_1}.

\subsection{When $a_{1,1} < a_{i,i}$ and $\alpha_{i} > \frac{1}{2}$}\label{subcase_2} Note that, by Lemma~\ref{simpler}, this situation arises only when $B_{1} > a_{i,i} - a_{1,1}$, and we need only consider $\alpha_{i} \leqslant \frac{B_{1}}{a_{i,i} - a_{1,1} + B_{1}}$. From \eqref{choice_of_j} and the hypothesis of Theorem~\ref{main_2}, for all sufficiently large $a_{i,i}$, we get:
\begin{align}\label{G_{j}_lower_bound_subcase_2}
\frac{G_{j}}{n} & \geqslant \frac{1-\alpha_{i}}{k-1} \left(\frac{1}{2} - \frac{1}{\sqrt{a_{i,i}}}\right) a_{i,i} - \frac{\alpha_{i} (a_{i,i} - a_{1,1})}{k-1} \nonumber\\
&\geqslant \frac{1-\alpha_{i}}{k-1} \left(\frac{1}{2} - \frac{1}{\sqrt{a_{i,i}}}\right) a_{i,i} - \frac{(1-\alpha_{i}) (a_{i,i} - a_{1,1} + B_{1}) \alpha_{i}}{k-1} \nonumber\\
&= \frac{1-\alpha_{i}}{k-1}\left\{\frac{a_{i,i}}{2} - \sqrt{a_{i,i}} - B_{1}\right\} \nonumber\\
&\geqslant \frac{1-\alpha_{i}}{k-1}\left\{\left(\frac{1}{2} - \epsilon_{i,1}\right) a_{i,i} - B_{1}\right\} > \frac{(1-\alpha_{i}) \epsilon_{i,1} a_{i,i}}{k-1},
\end{align}
where the last inequality follows from \eqref{not_all_intra_degrees_same_cond}. 

We again split the analysis into three parts depending on the ranges of values of $\alpha_{j}$ as given in \eqref{range_1}, \eqref{range_2} and \eqref{range_3}, with $c$ satisfying the following condition:
\begin{equation}\label{cond_c_subcase_2}
\log_{2}c > \max\left\{16k + 1, \frac{2 C(k-1)}{\epsilon_{i,1} 4^{8k}}\right\}.
\end{equation}
We first consider the case of $\eta_{j} \neq 0$ and then the case of $\eta_{j} = 0$  in each of these ranges.

When we are in the regime of \eqref{range_1} and $\eta_{j} \neq 0$, we note that the bounds in \eqref{p_s_lower_bound} and \eqref{each_deg_probab_lower_bound_subcase_1} hold, and therefore the same analysis as before goes through. When $\alpha_{j} \leqslant \frac{c}{a_{i,j}}$ and $\eta_{j} \neq 0$, the bounds in \eqref{p_s_upper_subcase_1} and \eqref{deg_upper_bound_subcase_1_subsub_2} hold for all $s = 1, \ldots, \left\lfloor\frac{\alpha_{i} n}{2}\right\rfloor$. Now, from \eqref{G_{j}_lower_bound_subcase_2}, we have:
\begin{equation}\label{alpha_{i}_lower_bound_subcase_2_subsub_2}
c \geqslant \alpha_{j} a_{i,j} \geqslant \frac{G_{j}}{n} \geqslant \frac{(1-\alpha_{i}) \epsilon_{i,1} a_{i,i}}{k-1} \quad \implies \quad \alpha_{i} \geqslant 1 - \frac{c(k-1)}{\epsilon_{i,1} a_{i,i}}.
\end{equation}
Then $\frac{1}{n} \log_{2} \Prob[A_{j}]$ can be bounded above by
\begin{align}
& -\frac{\log_{2}\left\{\alpha_{i} \alpha_{j} (1-\alpha_{i}) (1-\alpha_{j})\right\}}{2n} - \frac{\log_{2} n}{n} + H(\alpha_{i}) + H(\alpha_{j}) + \frac{\alpha_{i}}{2} \log_{2} \gamma \nonumber
\end{align}
of which the first two terms approach $0$ as $n \rightarrow \infty$. The remaining terms can be bounded above by
\begin{equation}
2H\left(\frac{c (k-1)}{\epsilon_{i,1} a_{i,i}}\right) + \frac{\log_{2} \gamma}{4}, \nonumber
\end{equation}
of which the first term can be made arbitrarily small by choosing $a_{i,i}$ sufficiently large, and the last term is a strictly negative constant. Hence the above is strictly negative for all $a_{i,i}$ sufficiently large.

Now, we consider $\eta_{j} = 0$ for the various regimes of $\alpha_{j}$. By the same reasoning as Remark~\ref{no_need_3}, the only interesting case is where $G_{j} < \frac{\alpha_{i} n}{2}$. First, we consider the range of $\alpha_{j}$ as in \eqref{range_1}. The bound of \eqref{each_deg_probab_lower_bound_subcase_1_eta_0} holds for all $s = 1, \ldots, G_{j}$. By \eqref{G_{j}_lower_bound_subcase_2} and since $\alpha_{j} \leqslant 1-\alpha_{i}$, we get 
\begin{equation}
\frac{G_{j}}{n} \geqslant \frac{\alpha_{j} \epsilon_{i,1} a_{i,i}}{k-1} > \frac{c \epsilon_{i,1} a_{i,i}}{(k-1) a_{i,j}} > \frac{2 \cdot 4^{8k} \epsilon_{i,1}}{C(k-1)}
\end{equation}
by \eqref{G_{j}_lower_bound_subcase_2}, \eqref{cond_c_subcase_2} and \eqref{intra_degrees_same}. Thus an upper bound on $\Prob[A_{j}]$ is given by 
\begin{equation}
O\left\{\frac{2^{H(\alpha_{i})n + H(\alpha_{j})n}}{\sqrt{\alpha_{i} \alpha_{j} (1-\alpha_{i}) (1-\alpha_{j})} n} \left(\frac{1}{c}\right)^{G_{j}/2}\right\} \leqslant O\left\{\frac{2^{2n}}{\sqrt{\alpha_{i} \alpha_{j} (1-\alpha_{i}) (1-\alpha_{j})} n} \left(\frac{1}{c}\right)^{\frac{4^{8k} \epsilon_{i,1} n}{C(k-1)}}\right\} \nonumber
\end{equation}
which is $o(n^{-1})$ for $c$ as in \eqref{cond_c_subcase_2}. 

Next, we consider $\alpha_{j}$ as in \eqref{range_2}. Again, the bounds in \eqref{p_s_upper_subcase_1} and \eqref{deg_upper_bound_subcase_1_subsub_2} hold for all $s = 1, \ldots, G_{j}$. From \eqref{G_{j}_lower_bound_subcase_2}, we get the following upper bound on $\frac{1}{n} \log_{2} \Prob[A_{j}]$:
\begin{equation}
-\frac{\log_{2}\left\{\alpha_{i} \alpha_{j} (1-\alpha_{i}) (1-\alpha_{j})\right\}}{2n} - \frac{\log_{2} n}{n} + 2H(\alpha_{i}) + \frac{(1-\alpha_{i}) \epsilon_{i,1} a_{i,i}}{k-1} \log_{2} \gamma. \nonumber
\end{equation}
The sum of the last two terms can be bounded above by
\begin{align}
-4(1-\alpha_{i}) \log_{2} (1-\alpha_{i}) + \frac{(1-\alpha_{i}) \epsilon_{i,1} a_{i,i}}{k-1} \log_{2} \gamma &< 8 (1-\alpha_{i})  \log_{2} a_{i,j} + \frac{(1-\alpha_{i}) \epsilon_{i,1} a_{i,i}}{k-1} \log_{2} \gamma \nonumber\\
&\leqslant (1-\alpha_{i}) \left\{8\log_{2} C + 8\log_{2} a_{i,i} + \frac{\epsilon_{i,1} a_{i,i}}{k-1} \log_{2} \gamma\right\}, \nonumber
\end{align}
which is strictly negative for all $a_{i,i}$ sufficiently large. 

Finally, we consider $\alpha_{j}$ in the range given in \eqref{range_3}. The bound in \eqref{deg_upper_bound_subcase_1_subsub_3} holds for all $s = 1, \ldots, G_{j}$. From \eqref{G_{j}_lower_bound_subcase_2}, we have
\begin{equation}\label{1-alpha_{i}_upper_bound_subcase_2}
\frac{1}{a_{i,j}} \geqslant \alpha_{j} a_{i,j} \geqslant \frac{G_{j}}{n} \geqslant \frac{(1-\alpha_{i}) \epsilon_{i,1} a_{i,i}}{k-1} \quad \implies \quad \alpha_{i} \geqslant 1 - \frac{k-1}{\epsilon_{i,1} a_{i,i} a_{i,j}}. 
\end{equation}
For any fixed positive integer $r > 2$ and all $a_{i,i}$ sufficiently large,
\begin{equation}
a_{i,i} > C^{\frac{1}{r-2}} \left(\frac{k-1}{\epsilon_{i,1}}\right)^{\frac{r-1}{r-2}} 2^{\frac{r}{r-2}}, \nonumber
\end{equation}
and by the same reasoning as in \eqref{log_upper_bound_subcase_1_subsub_3}, using \eqref{1-alpha_{i}_upper_bound_subcase_2} we conclude that
\begin{equation}
\log_{2} \left\{2 (1-\alpha_{i}) a_{i,j}\right\} \leqslant \frac{\log_{2}(1-\alpha_{i})}{r}.\nonumber
\end{equation}
An upper bound on $\frac{1}{n} \log_{2} \Prob[A_{j}]$ is given by
\begin{equation}
-\frac{\log_{2}\left\{\alpha_{i} \alpha_{j} (1-\alpha_{i}) (1-\alpha_{j})\right\}}{2n} - \frac{\log_{2} n}{n} + 2H(\alpha_{i}) + \frac{(1-\alpha_{i}) \epsilon_{i,1} a_{i,i}}{k-1} \log_{2}\left\{2 (1-\alpha_{i}) a_{i,j}\right\} \nonumber
\end{equation}
of which the first two terms approach $0$ as $n \rightarrow \infty$, and the sum of the last two terms can be bounded by
\begin{align}
-4(1-\alpha_{i}) \log_{2}(1-\alpha_{i}) + \frac{(1-\alpha_{i}) \epsilon_{i,1} a_{i,i}}{k-1} \frac{\log_{2} (1-\alpha_{i})}{r}
= (1-\alpha_{i}) \log_{2} (1-\alpha_{i}) \left\{-4 + \frac{\epsilon_{i,1} a_{i,i}}{r (k-1)}\right\} \nonumber
\end{align}
which is strictly negative for all $a_{i,i}$ sufficiently large. This brings us to the end of \S\ref{subcase_2}.

\subsection{When $a_{i,i} > a_{1,1}$ and $\alpha_{i} \leqslant \frac{1}{2}$}\label{subcase_3} By \eqref{choice_of_j}, we have
\begin{align}\label{G_{j}_lower_bound_subcase_3}
\frac{G_{j}}{n} &\geqslant \frac{\alpha_{i}}{k-1}\left(\frac{1}{2} - \frac{1}{\sqrt{a_{i,i}}}\right) a_{i,i} - \frac{\alpha_{i} (a_{i,i} - a_{1,1})}{k-1} \nonumber\\
&\geqslant \frac{\alpha_{i}}{k-1}\left\{a_{1,1} - \frac{a_{i,i}}{2} - \sqrt{a_{i,i}}\right\} \geqslant \frac{\alpha_{i} \delta_{1,i} a_{i,i}}{k-1},
\end{align}
by \eqref{not_all_intra_degrees_same_cond'}. Note that, if $\alpha_{j} \leqslant \frac{c}{a_{i,j}}$ for any constant $c > 1$, then 
\begin{equation}
c \geqslant \alpha_{j} a_{i,j} \geqslant \frac{G_{j}}{n} \geqslant \frac{\alpha_{i} \delta_{1,i} a_{i,i}}{k-1} \quad \implies \quad \alpha_{i} \leqslant \frac{c (k-1)}{\delta_{1,i} a_{i,i}},\nonumber
\end{equation}
which is strictly less than $\frac{1}{k}$ for all $a_{i,i}$ sufficiently large, contradicting our choice of $i$. Hence we need only consider the range of \eqref{range_1} for values of $\alpha_{j}$.

When $\eta_{j} \neq 0$, the argument is the same as the corresponding case in \S\ref{subcase_1}. When $\eta_{j} = 0$, the bound in \eqref{each_deg_probab_lower_bound_subcase_1_eta_0} holds, and from \eqref{G_{j}_lower_bound_subcase_3} and $\alpha_{i} \geqslant \frac{1}{k}$, we get the following upper bound on $\frac{1}{n} \log_{2} \Prob[A_{j}]$:
\begin{align}
& -\frac{\log_{2}\left\{\alpha_{i} \alpha_{j} (1-\alpha_{i}) (1-\alpha_{j})\right\}}{2n} - \frac{\log_{2} n}{n} + H(\alpha_{i}) + H(\alpha_{j}) - \frac{\delta_{1,i} a_{i,i}}{2 k (k-1)} \log_{2} c, \nonumber
\end{align}
and as $H(\alpha_{i}) + H(\alpha_{j}) \leqslant 2$, the above expression is strictly negative for all $a_{i,i}$ sufficiently large.

\bibliography{rsbm}

\begin{thebibliography}{48}
\providecommand{\natexlab}[1]{#1}
\providecommand{\url}[1]{\texttt{#1}}
\expandafter\ifx\csname urlstyle\endcsname\relax
  \providecommand{\doi}[1]{doi: #1}\else
  \providecommand{\doi}{doi: \begingroup \urlstyle{rm}\Url}\fi

\bibitem[Abbe(2017)]{survey}
Emmanuel Abbe.
\newblock Community detection and stochastic block models: recent developments.
\newblock \emph{The Journal of Machine Learning Research}, 18\penalty0
  (1):\penalty0 6446--6531, 2017.

\bibitem[Abbe and Sandon(2015)]{abbe_sandon}
Emmanuel Abbe and Colin Sandon.
\newblock Detection in the stochastic block model with multiple clusters: proof
  of the achievability conjectures, acyclic bp, and the information-computation
  gap.
\newblock \emph{arXiv preprint arXiv:1512.09080}, 2015.

\bibitem[Abbe et~al.(2015)Abbe, Bandeira, and Hall]{exact}
Emmanuel Abbe, Afonso~S Bandeira, and Georgina Hall.
\newblock Exact recovery in the stochastic block model.
\newblock \emph{IEEE Transactions on Information Theory}, 62\penalty0
  (1):\penalty0 471--487, 2015.

\bibitem[Barucca(2017)]{barucca}
Paolo Barucca.
\newblock Spectral partitioning in equitable graphs.
\newblock \emph{Physical Review E}, 95\penalty0 (6):\penalty0 062310, 2017.

\bibitem[Bender(1974)]{07}
Edward~A Bender.
\newblock The asymptotic number of non-negative integer matrices with given row
  and column sums.
\newblock \emph{Discrete Mathematics}, 10\penalty0 (2):\penalty0 217--223,
  1974.

\bibitem[BLONDELVD et~al.(2008)]{fast_unfolding}
GUILLAUMEJL BLONDELVD et~al.
\newblock Fast unfolding of community hierarchies in large networks, 2008.

\bibitem[Bollob{\'a}s and B{\'e}la(2001)]{08}
B{\'e}la Bollob{\'a}s and Bollob{\'a}s B{\'e}la.
\newblock \emph{Random graphs}.
\newblock Number~73. Cambridge university press, 2001.

\bibitem[Bordenave et~al.(2015)Bordenave, Lelarge, and
  Massouli{\'e}]{bordenave}
Charles Bordenave, Marc Lelarge, and Laurent Massouli{\'e}.
\newblock Non-backtracking spectrum of random graphs: community detection and
  non-regular ramanujan graphs.
\newblock In \emph{2015 IEEE 56th Annual Symposium on Foundations of Computer
  Science}, pages 1347--1357. IEEE, 2015.

\bibitem[Brito et~al.(2016)Brito, Dumitriu, Ganguly, Hoffman, and Tran]{09}
Gerandy Brito, Ioana Dumitriu, Shirshendu Ganguly, Christopher Hoffman, and
  Linh~V Tran.
\newblock Recovery and rigidity in a regular stochastic block model.
\newblock In \emph{Proceedings of the twenty-seventh annual ACM-SIAM symposium
  on Discrete algorithms}, pages 1589--1601. Society for Industrial and Applied
  Mathematics, 2016.

\bibitem[Brito et~al.(2018)Brito, Dumitriu, and Harris]{05}
Gerandy Brito, Ioana Dumitriu, and Kameron~Decker Harris.
\newblock Spectral gap in random bipartite biregular graphs and applications.
\newblock \emph{arXiv preprint arXiv:1804.07808}, 2018.

\bibitem[Choudhury and Paul(2013)]{choudhury_paul}
Deepjyoti Choudhury and Arnab Paul.
\newblock Community detection in social networks: an overview.
\newblock \emph{International Journal of Research in Engineering and
  Technology}, 2\penalty0 (2):\penalty0 6--13, 2013.

\bibitem[Coja-Oghlan(2010)]{coja}
Amin Coja-Oghlan.
\newblock Graph partitioning via adaptive spectral techniques.
\newblock \emph{Combinatorics, Probability and Computing}, 19\penalty0
  (2):\penalty0 227--284, 2010.

\bibitem[Condon and Karp(2001)]{condon}
Anne Condon and Richard~M Karp.
\newblock Algorithms for graph partitioning on the planted partition model.
\newblock \emph{Random Structures \& Algorithms}, 18\penalty0 (2):\penalty0
  116--140, 2001.

\bibitem[Decelle et~al.(2011{\natexlab{a}})Decelle, Krzakala, Moore, and
  Zdeborov{\'a}]{decelle_1}
Aurelien Decelle, Florent Krzakala, Cristopher Moore, and Lenka Zdeborov{\'a}.
\newblock Asymptotic analysis of the stochastic block model for modular
  networks and its algorithmic applications.
\newblock \emph{Physical Review E}, 84\penalty0 (6):\penalty0 066106,
  2011{\natexlab{a}}.

\bibitem[Decelle et~al.(2011{\natexlab{b}})Decelle, Krzakala, Moore, and
  Zdeborov{\'a}]{decelle_2}
Aurelien Decelle, Florent Krzakala, Cristopher Moore, and Lenka Zdeborov{\'a}.
\newblock Inference and phase transitions in the detection of modules in sparse
  networks.
\newblock \emph{Physical Review Letters}, 107\penalty0 (6):\penalty0 065701,
  2011{\natexlab{b}}.

\bibitem[Dyer and Frieze(1989)]{dyer_frieze}
Martin~E. Dyer and Alan~M. Frieze.
\newblock The solution of some random np-hard problems in polynomial expected
  time.
\newblock \emph{Journal of Algorithms}, 10\penalty0 (4):\penalty0 451--489,
  1989.

\bibitem[Fani and Bagheri(2017)]{fani_bagheri}
Hossein Fani and Ebrahim Bagheri.
\newblock Community detection in social networks.
\newblock \emph{Encyclopedia with Semantic Computing and Robotic Intelligence},
  1\penalty0 (01):\penalty0 1630001, 2017.

\bibitem[Friedman(2008)]{10}
Joel Friedman.
\newblock \emph{A proof of Alon's second eigenvalue conjecture and related
  problems}.
\newblock American Mathematical Soc., 2008.

\bibitem[Garcia et~al.(2018)Garcia, Ashourvan, Muldoon, Vettel, and
  Bassett]{brain}
Javier~O Garcia, Arian Ashourvan, Sarah Muldoon, Jean~M Vettel, and Danielle~S
  Bassett.
\newblock Applications of community detection techniques to brain graphs:
  Algorithmic considerations and implications for neural function.
\newblock \emph{Proceedings of the IEEE}, 106\penalty0 (5):\penalty0 846--867,
  2018.

\bibitem[Girvan and Newman(2002)]{girvan_newman}
Michelle Girvan and Mark~EJ Newman.
\newblock Community structure in social and biological networks.
\newblock \emph{Proceedings of the national academy of sciences}, 99\penalty0
  (12):\penalty0 7821--7826, 2002.

\bibitem[Heimlicher et~al.(2012)Heimlicher, Lelarge, and
  Massouli{\'e}]{labeled_2}
Simon Heimlicher, Marc Lelarge, and Laurent Massouli{\'e}.
\newblock Community detection in the labelled stochastic block model.
\newblock \emph{arXiv preprint arXiv:1209.2910}, 2012.

\bibitem[Holland et~al.(1983)Holland, Laskey, and Leinhardt]{holland}
Paul~W Holland, Kathryn~Blackmond Laskey, and Samuel Leinhardt.
\newblock Stochastic blockmodels: First steps.
\newblock \emph{Social networks}, 5\penalty0 (2):\penalty0 109--137, 1983.

\bibitem[Jerrum and Sinclair(1989)]{11}
Mark Jerrum and Alistair Sinclair.
\newblock Approximating the permanent.
\newblock \emph{SIAM journal on computing}, 18\penalty0 (6):\penalty0
  1149--1178, 1989.

\bibitem[Jerrum and Sorkin(1998)]{jerrum_sorkin}
Mark Jerrum and Gregory~B Sorkin.
\newblock The metropolis algorithm for graph bisection.
\newblock \emph{Discrete Applied Mathematics}, 82\penalty0 (1-3):\penalty0
  155--175, 1998.

\bibitem[Karrer and Newman(2011)]{karrer_newman}
Brian Karrer and Mark~EJ Newman.
\newblock Stochastic blockmodels and community structure in networks.
\newblock \emph{Physical review E}, 83\penalty0 (1):\penalty0 016107, 2011.

\bibitem[Lawler and Sokal(1988)]{12}
Gregory~F Lawler and Alan~D Sokal.
\newblock Bounds on the $l^2$ spectrum for markov chains and markov processes:
  a generalization of cheeger’s inequality.
\newblock \emph{Transactions of the American mathematical society},
  309\penalty0 (2):\penalty0 557--580, 1988.

\bibitem[Lelarge et~al.(2015)Lelarge, Massouli{\'e}, and Xu]{labeled_1}
Marc Lelarge, Laurent Massouli{\'e}, and Jiaming Xu.
\newblock Reconstruction in the labelled stochastic block model.
\newblock \emph{IEEE Transactions on Network Science and Engineering},
  2\penalty0 (4):\penalty0 152--163, 2015.

\bibitem[Leskovec et~al.(2008)Leskovec, Lang, Dasgupta, and
  Mahoney]{statistical_SBM}
Jure Leskovec, Kevin~J Lang, Anirban Dasgupta, and Michael~W Mahoney.
\newblock Statistical properties of community structure in large social and
  information networks.
\newblock In \emph{Proceedings of the 17th international conference on World
  Wide Web}, pages 695--704, 2008.

\bibitem[Levin et~al.(2017)Levin, Peres, and Wilmer]{06}
David~A.\ Levin, Yuval Peres, and Elizabeth~L. Wilmer.
\newblock \emph{Markov chains and mixing times}, volume 107.
\newblock American Mathematical Soc., 2017.

\bibitem[Massouli{\'e}(2014)]{massoulie_ramanujan}
Laurent Massouli{\'e}.
\newblock Community detection thresholds and the weak ramanujan property.
\newblock In \emph{Proceedings of the forty-sixth annual ACM symposium on
  Theory of computing}, pages 694--703. ACM, 2014.

\bibitem[M’barek et~al.(2018)M’barek, Borgi, Bedhiafi, and Hmida]{genetic}
Marwa~Ben M’barek, Amel Borgi, Walid Bedhiafi, and Sana~Ben Hmida.
\newblock Genetic algorithm for community detection in biological networks.
\newblock \emph{Procedia Computer Science}, 126:\penalty0 195--204, 2018.

\bibitem[McKay and Wang(2003)]{02}
Brendan~D McKay and Xiaoji Wang.
\newblock Asymptotic enumeration of 0--1 matrices with equal row sums and equal
  column sums.
\newblock \emph{Linear algebra and its applications}, 373:\penalty0 273--287,
  2003.

\bibitem[McKay and Wormald(1991)]{01}
Brendan~D McKay and Nicholas~C Wormald.
\newblock Asymptotic enumeration by degree sequence of graphs with degreeso (n
  1/2).
\newblock \emph{Combinatorica}, 11\penalty0 (4):\penalty0 369--382, 1991.

\bibitem[McSherry(2001)]{mcsherry}
Frank McSherry.
\newblock Spectral partitioning of random graphs.
\newblock In \emph{Proceedings 42nd IEEE Symposium on Foundations of Computer
  Science}, pages 529--537. IEEE, 2001.

\bibitem[Moosavi et~al.(2017)Moosavi, Jalali, Misaghian, Shamshirband, and
  Anisi]{moosavi}
Seyed~Ahmad Moosavi, Mehrdad Jalali, Negin Misaghian, Shahaboddin Shamshirband,
  and Mohammad~Hossein Anisi.
\newblock Community detection in social networks using user frequent pattern
  mining.
\newblock \emph{Knowledge and Information Systems}, 51\penalty0 (1):\penalty0
  159--186, 2017.

\bibitem[Mossel et~al.(2012)Mossel, Neeman, and Sly]{mossel_1}
Elchanan Mossel, Joe Neeman, and Allan Sly.
\newblock Stochastic block models and reconstruction.
\newblock \emph{arXiv preprint arXiv:1202.1499}, 2012.

\bibitem[Mossel et~al.(2014)Mossel, Neeman, and Sly]{mossel_3}
Elchanan Mossel, Joe Neeman, and Allan Sly.
\newblock Belief propagation, robust reconstruction and optimal recovery of
  block models.
\newblock In \emph{Conference on Learning Theory}, pages 356--370, 2014.

\bibitem[Mossel et~al.(2015)Mossel, Neeman, and Sly]{mossel_4}
Elchanan Mossel, Joe Neeman, and Allan Sly.
\newblock Consistency thresholds for the planted bisection model.
\newblock In \emph{Proceedings of the forty-seventh annual ACM symposium on
  Theory of computing}, pages 69--75, 2015.

\bibitem[Mossel et~al.(2018)Mossel, Neeman, and Sly]{mossel_2}
Elchanan Mossel, Joe Neeman, and Allan Sly.
\newblock A proof of the block model threshold conjecture.
\newblock \emph{Combinatorica}, 38\penalty0 (3):\penalty0 665--708, 2018.

\bibitem[Nadakuditi and Newman(2012)]{nadakuditi_newman}
Raj~Rao Nadakuditi and Mark~EJ Newman.
\newblock Graph spectra and the detectability of community structure in
  networks.
\newblock \emph{Physical review letters}, 108\penalty0 (18):\penalty0 188701,
  2012.

\bibitem[Newman(2004{\natexlab{a}})]{newman_1}
Mark~EJ Newman.
\newblock Detecting community structure in networks.
\newblock \emph{The European Physical Journal B}, 38\penalty0 (2):\penalty0
  321--330, 2004{\natexlab{a}}.

\bibitem[Newman(2004{\natexlab{b}})]{newman_fast}
Mark~EJ Newman.
\newblock Fast algorithm for detecting community structure in networks.
\newblock \emph{Physical review E}, 69\penalty0 (6):\penalty0 066133,
  2004{\natexlab{b}}.

\bibitem[Newman(2006)]{newman_2}
Mark~EJ Newman.
\newblock Finding community structure in networks using the eigenvectors of
  matrices.
\newblock \emph{Physical review E}, 74\penalty0 (3):\penalty0 036104, 2006.

\bibitem[Newman and Martin(2014)]{newman_martin}
MEJ Newman and Travis Martin.
\newblock Equitable random graphs.
\newblock \emph{Physical Review E}, 90\penalty0 (5):\penalty0 052824, 2014.

\bibitem[Pal and Zhu(2019)]{pal_zhu}
Soumik Pal and Yizhe Zhu.
\newblock Community detection in the sparse hypergraph stochastic block model.
\newblock \emph{arXiv preprint arXiv:1904.05981}, 2019.

\bibitem[Peixoto(2019)]{peixoto}
Tiago~P Peixoto.
\newblock Bayesian stochastic blockmodeling.
\newblock \emph{Advances in network clustering and blockmodeling}, pages
  289--332, 2019.

\bibitem[Rahiminejad et~al.(2019)Rahiminejad, Maurya, and
  Subramaniam]{metabolic}
Sara Rahiminejad, Mano~R Maurya, and Shankar Subramaniam.
\newblock Topological and functional comparison of community detection
  algorithms in biological networks.
\newblock \emph{BMC bioinformatics}, 20\penalty0 (1):\penalty0 212, 2019.

\bibitem[Yan et~al.(2014)Yan, Shalizi, Jensen, Krzakala, Moore, Zdeborov{\'a},
  Zhang, and Zhu]{yan}
Xiaoran Yan, Cosma Shalizi, Jacob~E Jensen, Florent Krzakala, Cristopher Moore,
  Lenka Zdeborov{\'a}, Pan Zhang, and Yaojia Zhu.
\newblock Model selection for degree-corrected block models.
\newblock \emph{Journal of Statistical Mechanics: Theory and Experiment},
  2014\penalty0 (5):\penalty0 P05007, 2014.

\end{thebibliography}

\end{document}